\documentclass[12pt]{article}

\usepackage{amsthm,tikz,amssymb, latexsym, amsmath, amscd, amsfonts, array, graphicx, lmodern, slantsc,enumerate,stmaryrd,rotating}

\DeclareFontFamily{U}{mathx}{}
\DeclareFontShape{U}{mathx}{m}{n}{<-> mathx10}{}
\DeclareSymbolFont{mathx}{U}{mathx}{m}{n}
\DeclareMathAccent{\widehat}{0}{mathx}{"70}
\DeclareMathAccent{\widecheck}{0}{mathx}{"71}

\usepackage[english]{babel}        
\usepackage[all]{xy}
\CompileMatrices

\usepackage{color}

\def\F{\protect\operatorname{Conf}}
\def\DF{\protect\operatorname{DConf}}

\def\tilde{~}

\newtheorem{proposition}{Proposition}[section]
\newtheorem{corollary}[proposition]{Corollary}
\newtheorem{definition}[proposition]{Definition}
\newtheorem{theorem}[proposition]{Theorem}
\newtheorem{remark}[proposition]{Remark}
\newtheorem{example}[proposition]{Example}
\newtheorem{lemma}[proposition]{Lemma}

\begin{document}

\title{Abrams' stabilization theorem for no-$k$-equal configuration spaces on graphs}
\author{Omar Alvarado-Gardu\~no	and Jes\'us Gonz\'alez}
\date{\empty}

\maketitle

\begin{abstract}
For a graph $G$, let $\F(G,n)$ denote the classical configuration space of $n$ labelled points in $G$. Abrams introduced a cubical complex, denoted here by $\DF(G,n)$, sitting inside $\F(G,n)$ as a strong deformation retract provided $G$ is suitably subdivided. Using discrete Morse Theory techniques, we extend Abrams' result to the realm of configurations having no $k$-fold collisions.
\end{abstract}



\section{Introduction and main result}
The ordered configuration space of $n$ labelled pairwise-distinct particles on a topological space $X$, $$\F(X,n)=\{(x_1,\ldots,x_n)\in X^n\colon x_i\neq x_j \mbox{ for } i\neq j \},$$ is a central object of study in pure and applied mathematics. The case $X=\mathbb{C}$ is historically and theoretically important, as $\F(\mathbb{C},n)$ is a model for the classifying space of Artin's classical pure braid group on $n$ strands. The topology of $\F(X,n)$ is reasonably well understood when $X=M$ is a smooth manifold of dimension at least 2, due in part to the presence of the Fadell-Neuwirth fibrations $\F(M,n+1)\to\F(M,n)$. On the other hand, after the pioneering work by Abrams and Ghrist \cite{MR2701024,MR1903151}, the case of a 1-dimensional space $X=G$ ---a graph--- has got attention from topologists due to its relevance in geometric group theory, as $\F(G,n)$ is aspherical, as well as in robotics, where $\F(G,n)$ approximates the space of states of $n$ autonomous agents moving without collision along a $G$-shaped system of tracks.

\smallskip
More generally, for an integer $k\in\{2,3,\ldots,n\}$, the \emph{no-$k$-equal configuration space} of $n$ labelled particles on $X$ without $k$-fold collisions is the subspace $\F^{k}(X,n)$ of the product $X^n$ consisting of the $n$-tuples $(x_1,\ldots,x_n)$ for which no set $\{x_{i_1},\cdots,x_{i_k}\}$ with $1\leq i_1<i_2<\ldots<i_k\leq n$ is a singleton. In an early appearance, these spaces were used to model and study Borsuk-Ulam type results (\cite{MR0425949}). More recently, no-$k$-equal configuration spaces have found a wide range of applications in homotopy theory and related areas. Thus, $\F^{k}(\mathbb{R},n)$ is the complement of a $k$-generalized Coxeter arrangement of type $A_{n-1}$, so its (co)homology groups can be expressed combinatorially (\cite{MR1317619,MR2915650}). On the other hand, the work in \cite{MR3426383} shows that no-$k$-equal configuration spaces play a subtle role in the study of the limit of Goodwillie's tower of a space of no-$k$-self-intersecting immersions. Furthermore, a slight generalization of no-$k$-equal configuration spaces, where quality rather than just quantity of collisions is controlled, has been used in~\cite{dobrinskayaloops} to study the homology of loop spaces on polyhedral products with simply connected factors. Worth noticing is the fact that the inclusion $\F^{k}(X,n)\hookrightarrow X^n$ turns out to be a homotopy equivalence through a range that grows linearly with $k$ and the local homotopy dimension of $X$ (\cite[Theorem~1.2]{MR3572355}). 

\smallskip
Following Abrams and Ghrist's lead, we focus on no-$k$-equal configuration spaces on graphs, which are objects with rich topological properties and a wide variety of applications. For instance, for the connected graph $I$ with two vertices and a single edge, $\F^{k}(I,n)\simeq\F^{k}(\mathbb{R},n)$ was used in~\cite{MR1243770} in order to estimate the size and depth of an optimal decision tree that answers whether there are $k$-multiplicities among $n$ given real numbers, while $\F^{3}(I,n)$ was shown in \cite{MR1386845} to be a model for the classifying space of the planar version of Artin's pure braid group on $n$ stands. In more applied terms, no-$k$-equal configuration spaces on graphs can provide a natural model for motion planning problems in digital microfluidics, see~\cite[Section~2.4]{MR2605308}. 

\smallskip
The central result in this paper (Theorem \ref{StabThm} below) gives a discrete model for the homotopy type of no-$k$-equal configuration spaces on any finite connected graph $G$. Namely, starting from the canonical cell structure of $G$, consider the product cell structure on the $n$-th Cartesian product $G^n$. The subspace $\F^k(G,n)$ fails to be a subcomplex, so one takes the largest subcomplex $\DF^k(G,n)$ of $G^n$ contained in $\F^k(G,n)$. This construction was first introduced for $k=2$ in \cite{MR2701024}, where Abrams argues that a suitable subdivision condition on $G$ suffices to show that $\DF^2(G,n)$ sits inside $\F^2(G,n)$ as a strong deformation retract. A homological version of such a result was obtained in \cite{MR3597806} for $k>2$, but a full homotopy statement had remained open. Here we fill in the gap by proving the missing homotopy equivalence. This is attained by imposing a subdivision requirement (Definition \ref{sufficietlysubdivided} below) which is, in fact, much milder than the one imposed by Chettih in her homological analysis. We thus get a computational avenue for studying and putting to work the topology of no-$k$-equal configuration spaces on general graphs.

\begin{definition}\label{sufficietlysubdivided}
A graph $G$ (possibly, in principle, with edge-loops or multiple edges between pairs of vertices) is said to be \emph{$(k,n)$-sufficiently subdivided} (or just \emph{sufficiently subdivided,} when the parameters $k$ and $n$ are implicit), provided the following two conditions hold:
\begin{itemize}
    \item[$(S.1)$] Every path between two essential vertices $u\neq v$ touches at least $n-k+2$ vertices (including $u,v$).
    \item[$(S.2)$] Every essential cycle (i.e.~one that is not homotopically trivial) at a vertex $u$ touches at least $n-k+3$ vertices (including $u$).
\end{itemize}
\end{definition}

Note that the case $k=2$ in Definition \ref{sufficietlysubdivided}
agrees with the condition conjectured in page 12 of Abrams' Ph.D.~thesis as the best possible requirement implying stabilization.

We remark that, for the purposes of this work, an essential vertex $v$ of a graph $G$ is characterized by the condition $\deg(v)\neq2$, where $\deg(v)$ stands for the number of connected components obtained by removing $v$ from any suitably small neighborhood of $v$ in $G$. Equivalently, in combinatorial terms, $\deg(v)$ is the number of edges incident to $v$, in the understanding that a loop-edge at $v$ contributes by 2 in $\deg(v)$. For instance, the essential vertices of the graph
$$
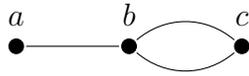
\begin{tikzpicture}
    \tikzset{Bullet/.style={fill = black, draw, color=#1, outer sep = 2, circle, minimum size = 1pt, scale = 0.5}}
    \node[Bullet=black, label=above:$a$] (d1) at (2.5,1) {};
    \node[Bullet=black, label=above:$b$] (e1) at (4,1) {};
    \node[Bullet=black, label=above:$c$] (f1) at (5.5,1) {};
    \path [-] (d1) edge (e1);
    \path [-] (e1) edge[bend left=40] (f1);
    \path [-] (e1) edge[bend right=40] (f1);
\end{tikzpicture}
$$
are $a$ and $b$, so the graph is not $(n-1,n)$-sufficiently subdivided, but becomes so after inserting at least one vertex on the edge $ab$, and at least two vertices on the essential loop at~$b$. On the other hand, any $(k,n)$-sufficiently subdivided graph with $n\geq k$ must be simple, i.e., has no edge-loops nor multiple edges, while any simple graph is $(n,n)$-sufficiently subdivided.

\begin{theorem}[Abrams' stabilization theorem for no-$k$-equal configuration spaces]\label{StabThm} For a $(k,n)$-sufficiently subdivided finite connected graph $G$, the inclusion $$\DF^k(G,n)\hookrightarrow\F^k(G,n)$$ is a homotopy equivalence.
\end{theorem}

As a result, the homotopy type of $\F^k(G,n)$ can be assessed through the discrete model $\DF^k(G,n)$ via techniques of discrete Morse theory. In particular, it is very promising to explore the existence and properties of a no-$k$-equal analogue of Farley-Sabalka discrete gradient field in \cite{MR2171804}. In fact, our proof of Theorem~\ref{StabThm} was motivated by the discrete Morse theoretic argument in \cite{MR3276733} that proves Abrams' stabilization theorem for $k=2$. We thank Safia Chettih for pointing us at the work of Prue and Scrimshaw.

\section{Subdivision} 
From now on $G=(V,E)$ stands for a finite connected graph with vertex set $V$ and edge set~$E$. As indicated in the introduction, we consider the product cell structure on $G^n$ coming from the standard cell structure on~$G$. As an abuse of notation, a cell $x=x_1\times\cdots\times x_n$ of $G^n$, where each $x_i$ is either a vertex or an edge, will generically be denoted as $x=(x_1,\ldots,x_n)$. Further, in what follows we will choose orientations of some of the edges of $G$. An oriented edge between vertices $u$ and $v$ will be referred to as $uv$ to indicate that the chosen orientation goes from $u$ to $v$, and we will use the notations $\iota(uv)=u$ and $\tau(uv)=v$ for the source and target of the oriented edge. Of course, the notation $uv$ is not sound under the presence of multiple edges, but for practical purposes this will never be the case (in view of the subdivision hypothesis for $G$ in Theorem \ref{StabThm}).

\smallskip
Fix vertices $u,v\in V$ of $G$ determining an edge $uv\in E$. The subdivided graph $G(w,uv)$, also denoted for short as $G'$ when the parameters $w$, $u$ and $v$ are implicit, is obtained from $G$ by inserting a vertex $w\not\in V$ on the edge $uv$, while the edge $uv$ is replaced by the couple of edges $wu$ and $wv$. Explicitly, $G(w,uv)$ has vertices $V':=V\cup \{w\}$ and edges $E':=(E-\{uv\})\cup \{wu, wv\}$.

\smallskip
The technical proof of the following key result lies at the core of this work:

\begin{theorem}[Stabilization of sufficiently subdivided discrete models]\label{StabDiscThm} In the situation above, there is a subcomplex $Y\hookrightarrow \DF^k(G'n)$ such that, as topological spaces, $Y\cong \DF^k(G,n)$. Furthermore, if $G$ is $(k,n)$-sufficiently subdivided, then $Y$ is a strong deformation retract of $\DF^k(G'n)$.
\end{theorem}

Informally, $\DF^k(G,n)$ is a strong deformation retract of $\DF^k(G',n)$ provided $G$ is $(k,n)$-sufficiently subdivided.

\begin{corollary}\label{weakmaintheorem}
If $G$ is $(k,n)$-sufficiently subdivided, then the inclusion
$$\iota\colon \DF^k(G,n)\hookrightarrow\F^k(G,n)$$
is a weak homotopy equivalence.
\end{corollary}
\begin{proof}
For an integer $j\geq0$, let $B_j$ stand for the iterated $j$-fold barycentric subdivision of~$G$, where $B_0=G$. The subspaces $$D_j:=\DF^k(B_j,n)$$ are compact and cover $\F^k(G,n)$ so, as observed in \cite[Example 1.3.9]{MR2456045}, $\F^k(G,n)$ carries the colimit topology determined by the sequence $$\cdots\subseteq D_j\subseteq D_{j+1}\subseteq\cdots.$$
In particular, \cite[Proposition 1.4.5]{MR2456045} yields that any compact subspace of $\F^k(G,n)$ is contained in some $D_j$, which implies in turn that $\iota_q\colon \pi_q(\DF^k(G,n))\rightarrow\pi_q(\F^k(G,n))$ is an isomorphism (with respect to any base point) for any $q\geq0$. For instance, if a homotopy class
\begin{equation}\label{enelnucleo}
[\alpha:S^q\rightarrow \DF^k(G,n)]\in \pi_q(\DF^k(G,n))
\end{equation}
is trivial on $\pi_q(\F^k(G,n))$, then a relevant based null-homotopy $S^q\times[0,1]\rightarrow\F^k(G,n)$ takes values in some $D_j$, so that the triviality of (\ref{enelnucleo}) follows from Theorem \ref{StabDiscThm}. Surjectivity of $\iota_q$ is proved similarly.
\end{proof}

\begin{proof}[Proof of Theorem \ref{StabThm}]
By Corollary \ref{weakmaintheorem} and Whitehead's theorem, it suffices to observe that $\F^k(G,n)$ has the homotopy type of a cell complex. 

Fix a linear order on the vertices of $G$, and consider the corresponding ordered simplicial structure on~$G^n$. Thus $V^n$ is the vertex set of $G^n$, while $d$-dimensional simplices are given by sets of cardinality $d+1$
\begin{equation}\label{lossimplejos}
\{(v_{0,1},\ldots,v_{0,n}), (v_{1,1},\ldots,v_{1,n}),\ldots,(v_{d,1},\ldots,v_{d,n})\}
\end{equation}
for vertices $v_{i,j}\in V$ such that $v_{0,j}\leq v_{1,j}\leq\cdots\leq v_{d,j}$ and $\{v_{0,j},v_{1,j},\ldots,v_{d,j}\}$ is a face of~$G$ for all $j$ (of course $v_{r,j}=v_{s,j}$ can hold for $r\neq s$). Then, as observed in \cite{MR3887191} for $k=2$, the $k$-fat diagonal $\Delta^n_{k} = G^n-\F^k(G,n)$ is given by the subcomplex of $G^n$ consisting of simplices (\ref{lossimplejos}) for which there are coordinates $1\leq j_1<j_2<\cdots<j_k\leq n$ with 
$$
(v_{0,j_1},v_{1,j_1},\cdots,v_{d,j_1})=(v_{0,j_2},v_{1,j_2},\cdots,v_{d,j_2})=\cdots=(v_{0,j_k},v_{1,j_k},\cdots,v_{d,j_k}).
$$
The desired observation, i.e.~the fact that $\F^k(G,n)=G^n-\Delta_k^n$ has the homotopy type of a cell (actually simplicial) complex , then follows from \cite[Lemma 4.2]{MR3887191} or, alternatively, from \cite[Lemma 70.1]{MR0755006} ---after a barycentric subdivision, to make sure that $\Delta_k^n$ is a full subcomplex of $G^n$.
\end{proof}

In constructing the subdivided graph $G(w,uv)$, it is convenient to distinguish the particular conditions under which the additional vertex $w$ is inserted to $G$. The following notation will be in force throughout the rest of the paper.

\begin{definition}[Primitive graph] \label{def:prim_graph}
 Given a connected graph $\Gamma$, the primitive graph $P({\Gamma})$ is constructed as follows:
\begin{itemize}
    \item[(a)] If ${\Gamma}$ is a cycle $C$, we select a vertex $u$ of $C$, and we remove all other vertices of $C$ keeping a single edge-loop at $u$.
    \item[(b)] If ${\Gamma}$ is not a cycle, we iteratively remove each degree-2 vertex $v$ in such a way that the given pair of edges $uv$ and $vw$ are replaced by a single (unsubdivided) edge $uw$.
\end{itemize}
\end{definition}
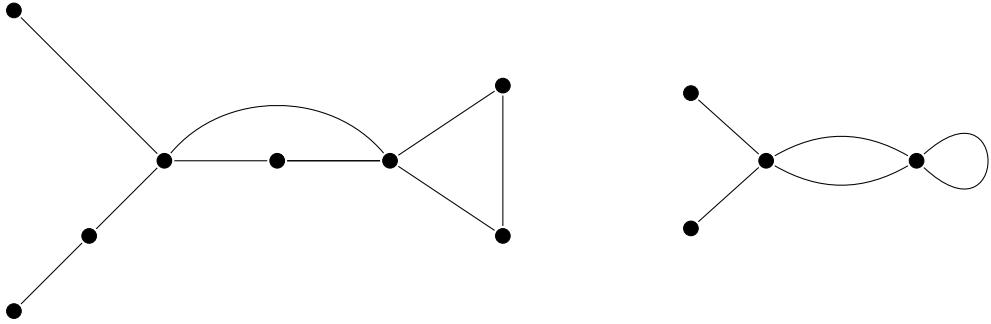
\begin{figure}
    \centering
\begin{tikzpicture}
    \tikzset{Bullet/.style={fill = black, draw, color=#1, outer sep = 2, circle, minimum size = 1pt, scale = 0.5}}
    \node[Bullet=black] (a) at (0,0) {};
    \node[Bullet=black] (b) at (1,1) {};
    \node[Bullet=black] (c) at (2,2) {};
    \node[Bullet=black] (d) at (0,4) {};
    \node[Bullet=black] (e) at (3.5,2) {};
    \node[Bullet=black] (f) at (5,2) {};
    \node[Bullet=black] (g) at (6.5,3) {};
    \node[Bullet=black] (h) at (6.5,1) {};

    \path [-] (a) edge (b);
    \path [-] (b) edge (c);
    \path [-] (c) edge (d);
    \path [-] (c) edge (e);
    \path [-] (e) edge (f);
    \path [-] (e) edge (f); 
    \path [-] (c) edge[bend left=50] (f); 
    \path [-] (f) edge (g);
    \path [-] (f) edge (h);
    \path [-] (h) edge (g);

    \node[Bullet=black] (pa) at (9,1.1) {};
    \node[Bullet=black] (pc) at (10,2) {};
    \node[Bullet=black] (pd) at (9,2.9) {};
    \node[Bullet=black] (pf) at (12,2) {};

    \path [-] (pa) edge (pc);
    \path [-] (pc) edge (pd);
    \path [-] (pc) edge[bend right=30] (pf); 
    \path [-] (pc) edge[bend left=30] (pf); 
    \path [-] (pf) edge[out=42,in=317,looseness=22] (pf);
\end{tikzpicture}
\caption{A graph ${\Gamma}$ (left) and its primitive graph $P({\Gamma})$ (right)}
\end{figure}

Vertices (edges) of $P({\Gamma})$ are called primitive vertices (edges) of ${\Gamma}$. Note that any non-primitive vertex of ${\Gamma}$ has degree 2, while the only situation where a primitive vertex has degree 2 holds when ${\Gamma}$ is a cycle, i.e., in item \emph{(a)} of Definition \ref{def:prim_graph}. Indeed, ${\Gamma}$ can be recovered from $P({\Gamma})$ through an iterative subdivision process $$P({\Gamma})={\Gamma}_0,{\Gamma}_1,\ldots,{\Gamma}_m={\Gamma}$$ with ${\Gamma}_\ell'={\Gamma}_{\ell+1}$. So, in what follows, we think of a primitive edge $\varepsilon$ of ${\Gamma}$ as the actual sequence of edges of ${\Gamma}$ that cover $\varepsilon$. Under this rationale, when we have a subdivided graph $G'=G(w,uv)$ of $G$, so that $P(G)=P(G')$, two scenarios arise depending on the primitive edge on which the new vertex $w$ gets inserted. Namely, either $w$ is inserted on a primitive edge between two different primitive vertices, or else, $w$ is inserted on a primitive loop. Note that, in the latter case, the loop can be based at a degree-2 vertex of $P(G)$ (when $G$ is a cycle). At the level of $G$, this means that either (A) $w$ is inserted on a combinatorial\footnote{Here, by a combinatorial path or cycle, we mean an ordered sequence of vertices such that any two consecutive ones support an edge.} path $P$ between two different primitive vertices, or else, (B) $w$ is inserted on a combinatorial cycle $C$ of $G$ that contains a single primitive vertex.

\smallskip
In case (A), let $l$ denote the number of vertices of $P$, so that the corresponding subdivided path $P'$ in $G'$ has $l+1$ vertices. We then label vertices so that
$$P=v_i,v_{i-1},\dots,v_1,v_{i+1},v_{i+2},\dots,v_l$$
where $v_i$ and $v_l$ are primitive, and the additional vertex $w$ is to be inserted in between $v_1$ and $v_{i+1}$ ($1\leq i\leq l-1$). Thus
$$P'=v_i,v_{i-1},\dots,v_1,w,v_{i+1},v_{i+2},\dots v_l$$
with the edge $a:=v_1v_{i+1}$ of $G$ replaced by the edges $wv_1$ and $wv_{i+1}$ of $G'$. Paths $P$ in $G$ and $P'$ in $G'$ are depicted below, where dotted segments represent possible edges other than those assembling $P$ or $P'$, some of which might connect $v_i$ to $v_l$ or even connect  one of these two vertices with itself ---if $G$ is not sufficiently subdivided.

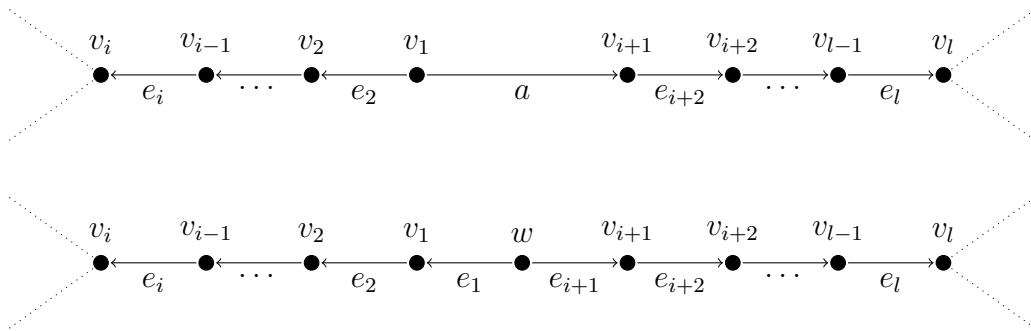
\begin{tikzpicture}
\tikzset{Bullet/.style={fill = black, draw, color=#1, outer sep = 2, circle, minimum size = 1pt, scale = 0.5}}
\node[Bullet=black,label=above: {$v_1$}] (pv1) at (-1.4,2.5){};
\node[Bullet=black,label=above: {$v_2$}] (pv2) at (-2.8,2.5){};
\node[Bullet=black,label=above: {$v_{i-1}$}] (pvim1) at (-4.2,2.5){};
\node[Bullet=black,label=above: {$v_i$}] (pvi) at (-5.6,2.5){};
\node[Bullet=black,label=above: {$v_{i+1}$}] (pvi1) at (1.4,2.5){};
\node[Bullet=black,label=above: {$v_{i+2}$}] (pvi2) at (2.8,2.5){};
\node[Bullet=black,label=above: {$v_{l-1}$}] (pvlm1) at (4.2,2.5){};
\node[Bullet=black,label=above: {$v_l$}] (pvl) at (5.6,2.5){};
\node (pextr1) at (-7, 3.5){};
\node (pextr2) at (-7, 1.5){};
\node (pextl1) at (7, 3.5){};
\node (pextl2) at (7, 2.5){};
\node (pextl3) at (7, 1.5){};

    \draw[->] (pv1)  -- node[below]{$a$} (pvi1);
    \draw[->] (pv1)  -- node[below]{$e_2$} (pv2);
    \draw[->] (pv2)  -- node[below]{$\dots $} (pvim1);
    \draw[->] (pvim1)  -- node[below]{$e_i$} (pvi);
    \draw[-, dotted] (pvi)  -- (pextr1);
    \draw[-, dotted] (pvi)  -- (pextr2);
    \draw[->] (pvi1)  -- node[below]{$e_{i+2}$} (pvi2);
    \draw[->] (pvi2)  -- node[below]{$\dots $} (pvlm1);
    \draw[->] (pvlm1)  -- node[below]{$e_l$} (pvl);
    \draw[-, dotted] (pvl)  -- (pextl1);
    \draw[-, dotted] (pvl)  -- (pextl3);
\node[Bullet=black,label=above: {$w$}] (pw) at (0,0){};
\node[Bullet=black,label=above: {$v_1$}] (pv1) at (-1.4,0){};
\node[Bullet=black,label=above: {$v_2$}] (pv2) at (-2.8,0){};
\node[Bullet=black,label=above: {$v_{i-1}$}] (pvim1) at (-4.2,0){};
\node[Bullet=black,label=above: {$v_i$}] (pvi) at (-5.6,0){};
\node[Bullet=black,label=above: {$v_{i+1}$}] (pvi1) at (1.4,0){};
\node[Bullet=black,label=above: {$v_{i+2}$}] (pvi2) at (2.8,0){};
\node[Bullet=black,label=above: {$v_{l-1}$}] (pvlm1) at (4.2,0){};
\node[Bullet=black,label=above: {$v_l$}] (pvl) at (5.6,0){};
    \node (pextr1) at (-7, 1){};
    \node (pextr2) at (-7, -1){};
    \node (pextl1) at (7, 1){};
    \node (pextl2) at (7, 0){};
    \node (pextl3) at (7, -1){};
    
    \draw[->] (pw)  -- node[below]{$e_1$} (pv1);
    \draw[->] (pv1)  -- node[below]{$e_2$} (pv2);
    \draw[->] (pv2)  -- node[below]{$\dots $} (pvim1);
    \draw[->] (pvim1)  -- node[below]{$e_i$} (pvi);
    \draw[-, dotted] (pvi)  -- (pextr1);
    \draw[-, dotted] (pvi)  -- (pextr2);
    \draw[->] (pw)  -- node[below]{$e_{i+1}$} (pvi1);    
    \draw[->] (pvi1)  -- node[below]{$e_{i+2}$} (pvi2);
    \draw[->] (pvi2)  -- node[below]{$\dots $} (pvlm1);
    \draw[->] (pvlm1)  -- node[below]{$e_l$} (pvl);
    \draw[-, dotted] (pvl)  -- (pextl1);
    \draw[-, dotted] (pvl)  -- (pextl3);
\end{tikzpicture}

In case (B), let $l-1$ denote the number of vertices of $C$, so that the corresponding subdivided cycle $C'$ in $G'$ has $l$ vertices. We then label vertices so that
$$C=v_i,v_{i-1},\dots,v_1,v_{i+1},v_{i+2},\dots,v_l$$
where $v_i=v_l$ is primitive, and the additional vertex $w$ is to be inserted in between $v_1$ and $v_{i+1}$ ($1\leq i\leq l-1$). Thus
$$C'=v_i,v_{i-1},\dots,v_1,w,v_{i+1},v_{i+2},\dots v_l$$
where, as before, the edge $a:=v_1v_{i+1}$ of $G$ is replaced by the edges $wv_1$ and $wv_{i+1}$ of $G'$. Cycles $C$ in $G$ and $C'$ in $G'$ are depicted below where, as in case (A), dotted segments represent possible edges not assembling $C$ or $C'$.

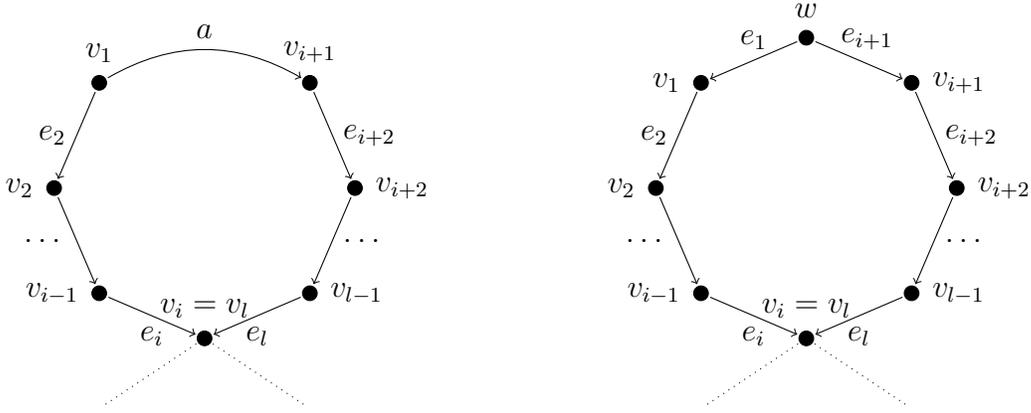
\begin{tikzpicture}
\tikzset{Bullet/.style={fill = black, draw, color=#1, outer sep = 2, circle, minimum size = 1pt, scale = 0.5}}
\node[Bullet=black,label=above: {$v_1$}] (cv1) at (-1.4,1.4){};
\node[Bullet=black,label=left: {$v_2$}] (cv2) at (-2,0){};
\node[Bullet=black,label=left: {$v_{i-1}$}] (cvim1) at (-1.4,-1.4){};
\node[Bullet=black,label=above: {$v_i=v_l$}] (cvi) at (0,-2){};
\node[Bullet=black,label=above: {$v_{i+1}$}] (cvi1) at (1.4,1.4){};
\node[Bullet=black,label=right: {$v_{i+2}$}] (cvi2) at (2,0){};
\node[Bullet=black,label=right: {$v_{l-1}$}] (cvlm1) at (1.4,-1.4){};
    
    \node (cext1) at (-1.5, -3){};
    \node (cext2) at (0, -3){};
    \node (cext3) at (1.5, -3){};
    
    \path[-, ->] (cv1)  edge[bend left=30] node[above]{$a$} (cvi1);
    \draw[->] (cv1)  -- node[left]{$e_2$} (cv2);
    \draw[->] (cv2)  -- node[left]{$\dots $} (cvim1);
    \draw[->] (cvim1)  -- node[below]{$e_i$} (cvi);
    \draw[->] (cvi1)  -- node[right]{$e_{i+2}$} (cvi2);
    \draw[->] (cvi2)  -- node[right]{$\dots $} (cvlm1);
    \draw[->] (cvlm1)  -- node[below]{$e_l$} (cvi);
    \draw[-, dotted] (cvi)  -- (cext1);
    \draw[-, dotted] (cvi)  -- (cext3);
\node[Bullet=black,label=above: {$w$}] (cw) at (8,2){};
\node[Bullet=black,label=left: {$v_1$}] (cv1) at (6.6,1.4){};
\node[Bullet=black,label=left: {$v_2$}] (cv2) at (6,0){};
\node[Bullet=black,label=left: {$v_{i-1}$}] (cvim1) at (6.6,-1.4){};
\node[Bullet=black,label=above: {$v_i=v_l$}] (cvi) at (8,-2){};
\node[Bullet=black,label=right: {$v_{i+1}$}] (cvi1) at (9.4,1.4){};
\node[Bullet=black,label=right: {$v_{i+2}$}] (cvi2) at (10,0){};
\node[Bullet=black,label=right: {$v_{l-1}$}] (cvlm1) at (9.4,-1.4){};
    
    \node (cext1) at (6.5, -3){};
    \node (cext2) at (8, -3){};
    \node (cext3) at (9.5, -3){};
    
    \draw[->] (cw)  -- node[above]{$e_1$} (cv1);
    \draw[->] (cv1)  -- node[left]{$e_2$} (cv2);
    \draw[->] (cv2)  -- node[left]{$\dots $} (cvim1);
    \draw[->] (cvim1)  -- node[below]{$e_i$} (cvi);
    \draw[->] (cw)  -- node[above]{$\,\,\,e_{i+1}$} (cvi1);    
    \draw[->] (cvi1)  -- node[right]{$e_{i+2}$} (cvi2);
    \draw[->] (cvi2)  -- node[right]{$\dots $} (cvlm1);
    \draw[->] (cvlm1)  -- node[below]{$e_l$} (cvi);
    \draw[-, dotted] (cvi)  -- (cext1);
    \draw[-, dotted] (cvi)  -- (cext3);
\end{tikzpicture}

In either case, we write $H$ ($H'$) to stand for either $P$ or $C$ ($P'$ or $C'$). Additionally, we orient the edges $e_s$ ($1\leq s\leq l$) assembling $H'$ as indicated in the figures above. Explicitly, $e_s:=v_{s-1}v_s$ for $s\not\in\{1,i+1\}$, while $e_1:=wv_1$ and $e_{i+1}:=wv_{i+1}$. Note that $e_i$ is an edge of $H$ for $s\not\in\{1,i+1\}$, while $e_1$ and $e_{i+1}$ assemble the edge $a$ of $H$. Furthermore,
\begin{equation}\label{cerrado}
\tau(e_i)=\tau(e_l), \text{\:\:\: i.e., \:\:\:} v_i=v_l
\end{equation}
 when $H=C$.

\smallskip
The following constructions and observations are freely used throughout the rest of the paper in order to organize the proof of Theorem\tilde\ref{StabDiscThm}.
\begin{definition}
For $v\in V$, $e\in E$ and $x=(x_1,\dots ,x_n)$ a cell of $G^n$, set
\begin{equation*} 
v(x):=\{1\leq t\leq n: x_t=v\}, \; e(x):=\{1\leq t\leq n: x_t=e\},
\; v(\overline{x}):=\{1\leq t\leq n: v\in \overline{x_t}\},
\end{equation*}
where $\overline{x_t}$ stands for the closure of the cell $x_t$. The cardinalities of $v(x)$, $e(x)$ and $v(\overline{x})$ are denoted respectively by $\eta(v,x)$, $\eta(e,x)$ and $\eta(v,\overline{x})$. Similarly, we get sets of coordinates $v'(x')$, $e'(x')$ and $v'(\overline{x'})$, and their respective cardinalities $\eta(v',x')$, $\eta(e',x')$ and $\eta(v',\overline{x'})$, for $v'\in V'$, $e'\in E'$ and $x'=(x'_1,\ldots,x'_n)$ a cell of $(G')^n$.
\end{definition}

If $y$ is a face of~$x$, then $v(\overline{y})\subseteq v(\overline{x})$ for each $v\in V$. Furthermore
\begin{equation}\label{pairwisedisj}
\{1,\dots ,n\}=\bigcup_{v\in V}v(x)\cup \bigcup_{e\in E}e(x)\text{\;\;\; and \;\;\;}v(\overline{x})=v(x)\cup \bigcup_{v\in \overline{e}}e(x)
\end{equation}
are both pairwise disjoint unions. (Note that, for a vertex $u$ and an edge $\varepsilon$, the notation $u\in\overline{\varepsilon}$ amounts to saying that $u$ is a vertex of $\varepsilon$.) Consequently, if $y$ is a face of $x$, then $\eta(v,\overline{y})\leq\eta(v,\overline{x})$ for every $v\in V$, while
\begin{equation}\label{varios1}
n=\sum_{v\in V}\eta(v,x)+\sum_{e\in E}\eta(e,x) \text{\:\:\:\:and\:\:\:\:} \eta(v,\overline{x})=\eta(v,x)+\sum_{v\in \overline{e}}\eta(e,x).
\end{equation}
Furthermore, since a non-empty intersection of the closures of any $k$ components of $x$ must contain a vertex, we see that a cell $x$ of\tilde$G^n$ is a cell of $\DF^k(G,n)$ if and only if 
\begin{equation}\label{nocolisions}
\eta(v,\overline{x})\leq k-1
\end{equation}
for every $v\in V$. Corresponding analogues of (\ref{pairwisedisj})--(\ref{nocolisions}) apply in the context of $G'$, where the notation $x'=(x_1',\ldots,x_n')$ is used for cells of $(G')^n$. In particular, we have equalities
$$
\eta(w,\overline{x'})=\eta(w,x')+\eta(e_1,x')+\eta(e_{i+1},x')\:\:\text{and}\:\:
\eta(v_s,\overline{x'})=\eta(v_s,x')+\eta(e_s,x')+\eta(e_{s+1},x')
$$
for $s\neq i,l$. Cases with $s\in\{i,l\}$ ---either in the context of $G$ or $G'$, as the graph structure of $G$ outside $H$ agrees with that of $G'$ outside $H'$--- require considering the collection
$$ext_s(x')\:\:=\bigcup_{e\in G'-H', \,\,v_s\in \overline{e}}e(x')$$
of coordinates $t\in\{1,2,\ldots,n\}$ such that $x'_t$ is an edge of $G'$ not in $H'$ and with $v_s\in \overline{x'_t}$. Thus, when $H'=P'$ (i.e.\tilde$H=P$) and $s\in\{i,l\}$, we have
$$
\eta(v_s,\overline{x'})=\eta(v_s,x')+\eta(e_s,x')+\eta(ext_s,x'),
$$
where $\eta(ext_s,x')$ stands for the cardinality of $ext_s(x')$. Similarly, when $H'=C'$ (i.e. $H=C$), so that $v_0:=v_i=v_l$ and $\eta(ext_0,x'):=\eta(ext_i,x')=\eta(ext_l,x')$, we have
$$
\eta(v_0,\overline{x'})=\eta(v_0,x')+\eta(e_i,x')+\eta(e_l,x')+\eta(ext_0,x').
$$

\begin{remark}{\em
As above, throughout the paper we will abuse notation referring to situations with $H'=P'$ (respectively $H'=C'$) simply as $H=P$ (respectively $H=C$). 
}\end{remark}

\section{The subcomplex}
The preliminary constructions in the previous section allow us to give a precise description of the subcomplex $Y$ of $\DF^k(G',n)$ in Theorem \ref{StabDiscThm}, and establish the asserted homeomorphism $Y\cong\DF^k(G,n)$.

\begin{definition}[Deflation/inflation]\label{deflainfla}
    \begin{itemize}
        \item[(a)] The \emph{deflation set} $\widecheck{x}$ of a cell $x=(x_1,\dots ,x_n)$ of $\DF^k(G,n)$ consists of all cells  $x'=(x'_1,\dots ,x'_n)$ of $(G')^n$ that can be obtained from $x$ by replacing each coordinate $x_t$ of $x$ having $x_t=a$ by one of the vertices $v_1, w, v_{i+1}$ or one of the edges $e_1, e_{i+1}$ (no relashionship is assumed among $x'_{t_1}$ and $x'_{t_2}$ when $x_{t_1}=a=x_{t_2}$, other than both $x'_{t_1}$ and $,x'_{t_2}$ must be chosen from $\{v_1, w, v_{i+1},e_1, e_{i+1}\})$.
        \item[(b)]  The \emph{inflation} of a cell $x'=(x'_1,\dots ,x'_n)$ of $\DF^k(G'n)$ is the cell $\widehat{x'}=(\widehat{x'}_1,\dots ,\widehat{x'}_n)$ given by:
        $$\widehat{x'}_j = \begin{cases}
            a, \text{ if } x'_j\in \{w, e_1, e_{i+1}\},\\
            x'_j, \text{ otherwise.}
        \end{cases}$$
    \end{itemize}
\end{definition}

\begin{example}\label{infdefexample}{\em
The deflation process in Definition \ref{deflainfla}\emph{(a)} is multi-valued, and every cell $x'\in \widecheck{x}$ is a cell of $\DF^k(G',n)$. On the other hand, the (unique) inflated cell in \emph{(b)} may fail to be a cell of $\DF^k(G,n)$ ---see Lemma \ref{ExtCell}. Yet, the inflation/deflation processes are in fact ``identity'' functions on cells that are suitably ``far'' from $w$. Namely, if a cell $x$ of $\DF^k(G,n)$ has $\eta(a,x)=0$, then the deflation set $\widecheck{x}$ is the singleton $\{x\}$, where $x$ is thought of as a cell of $\DF^k(G',n)$ satisfying $\eta(w,\overline{x'})=0$. In fact, by restricting to cells in
$$N_a:=\{x\in\DF^k(G,n)\colon \eta(a,x)=0\} \:\:\:\text{and}\:\:\: N_w:=\{x'\in\DF^k(G',n)\colon \eta(w,\overline{x'})=0\},$$
we see that the deflation/inflation processes are inverses of each other. 
}\end{example}

More generally, in Definition \ref{deflainfla}\emph{(a)}, the closures of the cells in $\widecheck{x}$ yield a subdivision of the closure of\tilde$x$. In particular, as spaces, $\DF^k(G,n)$ agrees with the subcomplex $Y$ of $\DF^k(G',n)$ whose collection of cells is the union of all the sets $\widecheck{x}$ as $x$ runs over all the cells of $\DF^k(G,n)$. 

\section{Preparing the deformation retract}
It remains to show that $Y$ is a strong deformation retract of $\DF^k(G',n)$ whenever $G$ is $(k,n)$-sufficiently subdivided. This is attained through Lemma \ref{MorseLemma} below, using discrete Morse theory, a technique used freely from now on. Discrete-Morse-theory details can be consulted in \cite{MR1612391}. 

\begin{lemma}\label{MorseLemma}
Let $W$ be a discrete gradient field on a finite regular complex $X$. If the $W$-critical cells form a subcomplex $Y$ of $X$, then $X$ strong deformation retracts onto $Y$ through simplicial collapses.
\end{lemma}
\begin{proof}
Choose a discrete Morse function $\phi:\text{Faces}(X)\rightarrow\mathbb{R}^{+}$ with gradient field $W$, and fix a constant $c\in\mathbb{R}^+$ such that
\begin{equation}\label{morsepunchline}
\phi(\alpha)<c
\end{equation}
for every critical cell $\alpha$. Then the formula
    $$f(\alpha)=\begin{cases}
        \phi(\alpha), &\text{ if } \alpha \text{ is critical;}\\
        \phi(\alpha)+c, &\text{ otherwise,}
    \end{cases}$$
defines a discrete Morse function $f:\text{Faces}(X)\rightarrow\mathbb{R}^+$ with gradient field $W$ as, in fact,
\begin{equation}\label{mismamorse}
        f(\alpha)<f(\beta)  \iff \phi(\alpha)<\phi(\beta)
\end{equation}
holds for any pair of cells $\alpha$ and $\beta$ of $X$ with $\dim(\alpha)+1=\dim(\beta)$ and $\alpha<\beta$. Indeed, if $\alpha$ is non-critical, then so is $\beta$, as critical cells form a subcomplex, and then (\ref{mismamorse}) is certainly satisfied. On the other hand, if $\alpha$ is critical, then the set 
$$\{\gamma \text{ cell of } X\colon \alpha < \gamma, \dim(\gamma)=\dim(\alpha)+1, \text{ and } \phi(\alpha)\geq \phi(\gamma)\}$$ 
is empty by definition. So $\phi(\alpha)<\phi(\beta)$, while $f(\alpha)<f(\beta)$ holds either because $\beta$ is critical or, else, by construction of $f$. Either way, we get (\ref{mismamorse}). 

The proof is complete in view of (\ref{morsepunchline}) and Forman's fundamental theorem in discrete Morse theory \cite[Theorem 3.3]{MR1612391}. Indeed, by construction of $f$, $f(\alpha)<c<f(\beta)$ holds whenever $\alpha$ is a critical and $\beta$ is non-critical. Thus $Y$ is the level subcomplex $X(c)$, while any cell with $f$-hight larger than $c$ is either redundant or collapsible, which yields $X\searrow X(c)=Y$. 
\end{proof}

In view of Lemma \ref{MorseLemma}, Theorem \ref{StabThm} will be proved once we construct a discrete gradient field on $\DF^k(G',n)$ whose critical cells are those of $Y$. To address such a goal, we start by giving a manageable description of the \emph{external cells}, i.e., cells of $\DF^k(G',n)$ that are not cells of $Y$, as those are the ones that have to be acyclically paired.

\begin{lemma}\label{ExtCell}
    A cell $x'$ of $\DF^k(G',n)$ is external if and only if $\widehat{x'}$ fails to be a cell of $\DF^k(G,n)$.
\end{lemma}
\begin{proof}
    If $\widehat{x'}$ is a cell of $\DF^k(G,n)$, then $x'$ cannot be external, as $x'$ is an element of the deflation of $\widehat{x'}$. Conversely, if $x'$ is not an external cell, say $x'\in\widecheck{x}$ for a cell $x$ of $\DF^k(G,n)$, then $\widehat{x'}$ must be a cell of $\DF^k(G,n)$, since, in fact, $\widehat{x'}$ is (perhaps not an immediate) face of\tilde$x$.
\end{proof}

The external-cell condition on $x'$ given in Lemma \ref{ExtCell} amounts to requiring that $\widehat{x'}$ has at least $k$ collisions. But such collisions are bound to hold at vertices $v_1$ or $v_{i+1}$. Thus $x'$ is external if and only if 
\begin{equation}\label{caracterizacion1}
\eta(v_1,\overline{\widehat{x'}})\geq k \text{\;\;\;\;or\;\;\;\;} \eta(v_{i+1},\overline{\widehat{x'}})\geq k.
\end{equation}
For our purposes, it is convenient to spell out conditions (\ref{caracterizacion1}) in terms of $x'$, rather than in terms of the inflated cell $\widehat{x'}$.

\begin{lemma}\label{spelled}
Conditions in (\ref{caracterizacion1}) can be spelled out, respectively, as
\begin{equation}\label{simplificadas}
\eta(w,\overline{x'})+\eta(v_1,\overline{x'})-\eta(e_1,x')\geq k 
\text{\:\:\: and \:\:\:}\eta(w,\overline{x'})+\eta(v_{i+1},\overline{x'})-\eta(e_{i+1},x')\geq k.
\end{equation}
\end{lemma}
\begin{proof}
Let $x'=(x'_1,\ldots,x_n')$ and $\widehat{x'}=(\widehat{x'}_1,\ldots,\widehat{x'}_n)$. For $s\in\{1,i+1\}$, the closure of a coordinate $\widehat{x'}_t$  touches\tilde$v_s$ if and only if the closure of $x'_t$ touches either $v_s$ or $w$, though cases with $x'_t=e_s$ are accounted for twice. Therefore $\eta(v_s,\overline{\widehat{x'}})=\eta(w,\overline{x'})+\eta(v_s,\overline{x'})-\eta(e_s,x')$.
\end{proof}

\section{The rank function}
\begin{definition}\begin{itemize}\item[(a)]
For $s\in\{1,2,\ldots,l\}$, a cell $x'=(x'_1,x'_2,\ldots,x'_n)\in\DF^k(G',n)$ is said to have property $R(s)$ when $\eta(v_s, \overline{x'})=\eta(e_s, x')$, i.e., provided any coordinate $x'_t$ whose closure touches $v_s$ must in fact be $x'_t=e_s$.
\item[(b)] The rank of an external cell $x'=(x'_1,\ldots,x'_n)$, denoted by $rank(x')$, is the minimal $s\in\{1,2,\ldots,l\}$ such that $x'$ has property $R(s)$. Thus, $rank(x')=j$ if and only if $\eta(v_j, \overline{x'})=\eta(e_j, x')$ but $\eta(v_s, \overline{x'})>\eta(e_s, x')$ for $1\leq s<j$.

\end{itemize}
 \end{definition}

The subdivision hypothesis for $G$ in Theorem \ref{StabThm} (in force from now on) ensures that the rank function is well defined. 
\begin{proposition} \label{RankWD}
Assume $G$ is $(k,n)$-sufficiently subdivided, and let 
$$\ell:=\begin{cases}\ \ l,& \text{if \ $H=P$;}\\
l-1,& \text{if \ $H=C$.}\end{cases}$$
Then, for each external cell $x'=(x'_1,\ldots,x'_n)$, the set $\{1\leq s\leq \ell\colon x' \text{ has property } R(s)\}$ is non-empty. In particular $rank(x')$ is well-defined and takes values in $\{1,2,\ldots,\ell\}$.
\end{proposition}
\begin{proof}
Assume for a contradiction that 
\begin{equation} \label{eq:diff}
\mbox{$\eta(v_s,\overline{x'})-\eta(e_s,x')\geq 1$, 
for $1\leq s\leq \ell$,}
\end{equation}
and let $S$ be the cardinality of the set $$\{t\in\{1,2,\ldots,n\}\colon x'_t\in\{w,v_1,e_1,v_2,e_2,\ldots, v_l,e_l\}\},$$ so that $S\leq n$. 

\smallskip\noindent\textbf{Case $H=P$ (so $\ell=l$):} We have
\begin{align}
S&=\eta(w, x')\:+\sum_{1\leq s\leq i-1}\left(\rule{0mm}{4mm}\eta(v_s, x')+\eta(e_s, x')\right) +\eta(v_{i}, x')+\eta(e_{i}, x')\nonumber\\
&\hspace{2cm} +\sum_{i+1\leq s\leq l-1}\left(\rule{0mm}{4mm}\eta(v_s, x')+\eta(e_s, x')\right) +\eta(v_{l}, x')+\eta(e_{l}, x')\nonumber\\
&=\eta(w, x')\hspace{.5mm}+\hspace{.5mm}\eta(e_1, x')+\eta(e_{i+1}, x')\nonumber\\
&\hspace{2cm} +\sum_{1\leq s\leq i-1}\left(\rule{0mm}{4mm}\eta(v_s, x')+\eta(e_{s+1}, x')\right) +\eta(v_{i}, x')\nonumber\\
&\hspace{2cm} +\sum_{i+1\leq s\leq l-1}\left(\rule{0mm}{4mm}\eta(v_s, x')+\eta(e_{s+1}, x')\right) +\eta(v_{l}, x')\nonumber\\
&=\eta(w, \overline{x'})+\sum_{\substack{1\leq s\leq l-1\\
s\neq i}}\left(\rule{0mm}{4mm}\eta(v_s, \overline{x'})-\eta(e_s, x')\right)+\eta(v_i, x')+\eta(v_l, x'),\label{dereferencia}
\end{align}
where the last equality follows from the fact that the entries whose closure touches $v_s$ are precisely $v_s$, $e_s$ and $e_{s+1}$ when $s\not\in\{i,l\}$.
    
If there were no edges between $v_i$ and $v_l$, so that $ext_i(x')\cap ext_l(x')=\varnothing$, then (\ref{eq:diff}), (\ref{dereferencia}) and the fact that the external cell $x'$ satisfies one of the inequalities in (\ref{simplificadas}) would yield
$$
n\geq S+\eta(ext_{i}, x')+\eta(ext_{l}, x')=\eta(w, \overline{x'})+\sum_{1\leq s\leq l}\left(\rule{0mm}{4mm}\eta(v_s, \overline{x'})-\eta(e_s, x')\right)\geq k+l-1,
$$
contradicting item ($S$.1) in Definition \ref{sufficietlysubdivided}. Therefore there must exist an edge joining $v_i$ and $v_l$, so that
\begin{equation}\label{cntdn3}
l\geq n-k+3,
\end{equation}
in view of item ($S$.2) in Definition \ref{sufficietlysubdivided}. On the other hand, for $\{s_1,s_2\}=\{i,l\}$, (\ref{dereferencia}) yields   
    \begin{equation}\label{cntdn2}
        n\geq S+\eta(ext_{s_1}, x')\geq \eta(w, \overline{x'})+\sum_{\substack{1\leq s\leq l\\
        s\neq s_2}}\left(\rule{0mm}{4mm}\eta(v_s, \overline{x'})-\eta(e_s, x')\right).
    \end{equation}
Choosing in fact
$$(s_1,s_2)=\begin{cases}
(i,l), & \text{provided $x'$ satisfies the first inequality in (\ref{simplificadas}) and $i=1$;} \\
(l,i), &\text{otherwise,}
\end{cases}$$
we get from (\ref{eq:diff}) that the right hand-side in (\ref{cntdn2}) ---and therefore $n$--- is bounded from below by $k+l-2$, which is inconsistent with (\ref{cntdn3}).

\smallskip\noindent\textbf{Case $H=C$ (so $\ell=l-1$):} 
Since $v_i=v_l$, we now have
    \begin{align}
        S
        &=\left(\rule{0mm}{4mm}\eta(w, x')+\eta(e_1, x')+\eta(e_{i+1}, x')\right)+\eta(v_{i}, x')\nonumber\\
        &\hspace{.5cm}+\sum_{1\leq s\leq i-1}\left(\rule{0mm}{4mm}\eta(v_s, x')+\eta(e_{s+1}, x')\right) +\sum_{i+1\leq s\leq l-1}\left(\rule{0mm}{4mm}\eta(v_s, x')+\eta(e_{s+1}, x')\right)\nonumber\\
        &=\eta(w, \overline{x'})+\eta(v_i, x')+\sum_{\substack{1\leq s\leq l-2\\
        s\neq i}}\left(\rule{0mm}{4mm}\eta(v_s, \overline{x'})-\eta(e_s, x')\right)+M(x'), \label{primitivo}
    \end{align}
where
$$
M(x'):=\begin{cases}
\eta(v_{l-1}, x')+\eta(e_{l}, x'), & \text{ if $i\leq l-2,$}\\ 0, & \text{ otherwise.}\end{cases}
$$
Using $n\geq S+\eta(ext_{i}, x')$, we then get
\begin{equation}\label{laestimacioncompleta}
n\geq\eta(w, \overline{x'})+\eta(v_i, x')+\sum_{\substack{1\leq s\leq l-2\\ s\neq i}}\left(\rule{0mm}{4mm}\eta(v_s, \overline{x'})-\eta(e_s, x')\right)+M(x')+\eta(ext_{i}, x').
\end{equation}
The estimations in (\ref{primitivo}) and (\ref{laestimacioncompleta}) are used next to argue that
\begin{equation}\label{cntdtn5}
n\geq k+l-3,
\end{equation}
which yields the desired contradiction (to item (S.2) in Definition \ref{sufficietlysubdivided}), as the cycle $C$ in $G$ has $l-1$ vertices.

\smallskip
Verifying (\ref{cntdtn5}) is easy when $i=1$ and the first inequality in (\ref{simplificadas}) holds, as (\ref{laestimacioncompleta}) yields
    \begin{align*}
        n&\geq \eta(w, \overline{x'})\;+\sum_{2\leq s\leq l-2}\left(\rule{0mm}{4mm}\eta(v_s, \overline{x'})-\eta(e_s, x')\right) + \left(\rule{0mm}{4mm}\eta(v_1,\overline{x'})-\eta(e_1,x')\right)\geq k + l-3,
    \end{align*}
in view of (\ref{eq:diff}). Likewise, when $i=l-1$ (so $v_{l-1}=v_l$) and the second inequality in (\ref{simplificadas}) holds, (\ref{laestimacioncompleta}) becomes
    \begin{align*}
        n&\geq\eta(w, \overline{x'})+\sum_{1\leq s\leq l-2}\left(\rule{0mm}{4mm}\eta(v_s, \overline{x'})-\eta(e_s, x')\right) + \left(\rule{0mm}{4mm}\eta(v_l,\overline{x'})-\eta(e_{l-1},x')-\eta(e_l,x')\right)\\
        &=\eta(w, \overline{x'})+\sum_{1\leq s\leq l-3}\left(\rule{0mm}{4mm}\eta(v_s, \overline{x'})-\eta(e_s, x')\right) + \left(\rule{0mm}{4mm}\eta(v_{l-2}, \overline{x'})-\eta(e_{l-2}, x')\right)\\
        &\hspace{7.45cm}+\left(\rule{0mm}{4mm}\eta(v_{l},\overline{x'})-\eta(e_{l-1},x')-\eta(e_l,x')\right)\\
        &=\eta(w, \overline{x'})+\sum_{1\leq s\leq l-3}\left(\rule{0mm}{4mm}\eta(v_s, \overline{x'})-\eta(e_s, x')\right) + \left(\rule{0mm}{4mm}\eta(v_{l-2}, x')+\eta(e_{l-1}, x')\right)\\
        &\hspace{7.45cm}+\left(\rule{0mm}{4mm}\eta(v_{l},\overline{x'})-\eta(e_{l-1},x')-\eta(e_l,x')\right)\\
        &\geq \eta(w, \overline{x'})+\sum_{1\leq s\leq l-3}\left(\rule{0mm}{4mm}\eta(v_s, \overline{x'})-\eta(e_s, x')\right) + \left(\rule{0mm}{4mm}\eta(v_{l}, \overline{x'})-\eta(e_{l}, x')\right)\\
        &\geq k + l-3,
    \end{align*}
again in view of (\ref{eq:diff}), thus giving (\ref{cntdtn5}).

\begin{remark}\label{valoresmeores}{\em
In order to get (\ref{cntdtn5}) in the two cases just analyzed, it suffices to assume\tilde(\ref{eq:diff}) for $1\leq s \leq l-2$ ---rather than $1\leq s \leq l-1$.
}\end{remark}

Back to the proof, since one of the inequalities in (\ref{simplificadas}) is satisfied, the only two cases remaining to be analyzed are \emph{(i)} when $x'$ satisfies the first inequality in (\ref{simplificadas}) with $i>1$, and \emph{(ii)} when $x'$ satisfies the second inequality in (\ref{simplificadas}) with $i<l-1$. But in these two situations, (\ref{primitivo}) gives
\begin{equation}\label{thetwosituations}
n\geq S\geq\eta(w, \overline{x'})\,+\sum_{\substack{1\leq s\leq l-1\\ s\neq i}}\left(\rule{0mm}{4mm}\eta(v_s, \overline{x'})-\eta(e_s, x')\right)\geq k +l-3,
\end{equation}
where we can use the required difference in the summation in order to apply (\ref{simplificadas}), while the remaining differences are summed up to apply (\ref{eq:diff}).
\end{proof}

We will need to know that, under some circumstances, the rank function takes values slightly smaller than what is asserted in the statement of Proposition \ref{RankWD}:

\begin{proposition} \label{RankRestC}
If $H=C$ and $i=l-1$, then $rank(x')\leq l-2$ for any external cell $x'$ of $\DF^k(G',n)$.
\end{proposition}
\begin{proof}
Assume for a contradiction that some external cell $x'$ has $rank(x')=i=l-1$. In particular,
\begin{align} \label{eq:diffC}
\eta(v_s, \overline{x'})-\eta(e_s, x')\geq 1, \quad \text{for }\; 1\leq s\leq i-1=l-2.    
\end{align}
By Remark \ref{valoresmeores}, we can safely assume that $x'$ satisfies the first inequality in\tilde(\ref{simplificadas}), so we are in the situation of case \emph{(i)} at the end of the previous proof. Moreover, since the summation in (\ref{thetwosituations}) runs over $1\leq s\leq l-2$, the last inequality in (\ref{thetwosituations}) can still be deduced from (\ref{eq:diffC}), leading to the required contradiction.
\end{proof}

We close this section with an observation playing a central role in the construction of the gradient field we need.

\begin{lemma}\label{paraelcampo}
Any rank-$j$ external cell $x'$ of $\DF^k(G',n)$ satisfies
$$\eta(\iota(e_j),x')+\eta(e_j,x')\geq 1.$$
\end{lemma}
\begin{proof}
Assume $j\not\in\{1, i+1\}$, so that $\iota(e_j)=v_{j-1}$. The fact that $rank(x')>j-1$ gives the inequality in
$\eta(e_{j-1},x')<\eta(v_{j-1},\overline{x'})=\eta(v_{j-1},x')+\eta(e_{j-1},x')+\eta(e_{j},x')$, which in turn yields the conclusion.

On the other hand, assume $j\in\{1, i+1\}$, so that $\iota(e_j)=w$, and let $j'$ be defined by $\{j'\}=\{1, i+1\}\setminus\{j\}$. The rank hypothesis gives $\eta(v_j,\overline{x'})-\eta(e_j,x')=0$, so that
\begin{equation}\label{contra1}
\eta(w,\overline{x'})+\eta(v_j,\overline{x'})-\eta(e_j,x')=\eta(w,\overline{x'})\leq k-1,
\end{equation}
as $x'$ is a cell of $\DF^k(G',n)$. Assume in addition (to draw a contradiction) that $\eta(w,x')=\eta(e_j,x')=0$, so that $\eta(w,\overline{x'})=\eta(e_{j'},x')$. Then 
\begin{equation}\label{contra2}
\eta(w,\overline{x'})+\eta(v_{j'},\overline{x'})-\eta(e_{j'},x')=\eta(v_{j'},\overline{x'})\leq k-1,
\end{equation}
again because $x'$ is a cell of $\DF^k(G',n)$. But (\ref{contra1}) and (\ref{contra2}) are incompatible with (\ref{simplificadas}), as $x'$ is external.
\end{proof}

\section{Construction of the pairing}
Throughout this section we assume that $G$ is $(k,n)$-sufficiently subdivided, so that the rank function is well-defined. The gradient field we need is constructed as a perfect matching of external cells. Actually, Proposition \ref{CorFaceOp} below allows us to construct a family $\{W_j\}_j$ of matchings, with $W_j$ perfectly pairing all external cells of rank $j$.

\begin{remark}\label{aclaracion}{\em
In what follows we deal with cells $x'=(x'_1,\ldots,x'_n)$ and $y'=(y'_1,\ldots,y'_n)$ of $(G')^n$ with $y'$ a codimension-1 face of $x'$. Say $x'_t=y'_t$ for $t\neq m$ and some $m\in\{1,2,\ldots,n\}$, while $y'_m$ is a vertex of the edge $x'_m$ of $G'$. Then, for a vertex or edge $\mu$ of $G'$, we often compare the set $\mu(y')$ to the corresponding set $\mu(x')$ and, in doing so, it is convenient to keep in mind that $\mu(y')\cap\left(\{1,2,\ldots,n\}\setminus\{m\}\rule{0mm}{4mm}\right)=\mu(x')\cap\left(\{1,2,\ldots,n\}\setminus\{m\}\rule{0mm}{4mm}\right)$. In other words, the only difference between $\mu(y')$ and $\mu(x')$ would be the potential fact that $m$ lies in one of these sets without lying on the other one. In particular, checking a possible equality $\mu(y')=\mu(x')$ amounts to checking that $m$ lies in $\mu(y')$ if and only if $m$ lies in $\mu(x')$. The use of these observations will be explicit in the following proof but, later in the paper, we will use the observations without further notice.
}\end{remark}

\begin{proposition}[Rank stability for ``$\iota$-type'' faces]\label{CorFaceOp}
Let $m$, $x'$ and $y'$ be as in Remark \ref{aclaracion}. If $x'_m=e_j$ and $y'_m=\iota(e_j)$ for some $j\in\{1,2,\ldots,l\}$, then $\eta(w,\overline{y'})=\eta(w,\overline{x'})$. Furthermore the equality
\begin{equation}\label{comparacion}
v_s(\overline{y'})-e_s(y')=v_s(\overline{x'})-e_s(x')
\end{equation}
holds for all $s\in\{1,2,\ldots,l\}$, except when $\{j,s\}=\{i,l\}$ and $H=C$.
\end{proposition}
\begin{proof}
The equality $\eta(w,\overline{y'})=\eta(w,\overline{x'})$ is elementary in view of Remark \ref{aclaracion}. Indeed,
$$
m\in\overline{w}(y')\Leftrightarrow y'_m=w\Leftrightarrow j\in\{1,i+1\}\Leftrightarrow x'_m\in\{e_1,e_{i+1}\}\Leftrightarrow m\in\overline{w}(x').
$$
On the other hand, for $z'=(z'_1,\ldots,z'_n)$, the set $v_s(\overline{z'})-e_s(z')$ consists of the coordinates $t\in\{1,2,\ldots,n\}$ with $v_s\in\overline{z'_t}$ and $z'_t\neq e_s$. So, as in Remark \ref{aclaracion}, (\ref{comparacion}) could fail only if $m$ lies on one of the sets in (\ref{comparacion}) without lying on the other. Now, $m\in v_s(\overline{y'})-e_s(y')$ if and only if
\begin{equation}\label{rangoestable1}
v_s=\iota(e_j)
\end{equation}
whereas $m\in v_s(\overline{x'})-e_s(x')$ if and only if
\begin{equation}\label{rangoestable2}
v_s\in\overline{e_j}\text{\:\:\: with \:\:\:}s\neq j.
\end{equation}
The result follows by noticing that (\ref{rangoestable1}) implies (\ref{rangoestable2}), whereas the only way in which (\ref{rangoestable2}) holds without (\ref{rangoestable1}) holding is precisely when $\{j,s\}=\{i,l\}$ and $H=C$.
\end{proof}

\begin{corollary} \label{FieldWD}
In Proposition \ref{CorFaceOp}, $x'$ is a rank-$j$ external cell of  $\DF(G',n)$ if and only if so is $y'$.
\end{corollary}
\begin{proof}
Assume either $x'$ or $y'$ is external of rank $j$. By Proposition \ref{CorFaceOp}, the only situation where the assertion
\begin{equation}\label{sisucede}
\eta(v_s,\overline{y'})-\eta(e_s,y')=\eta(v_s,\overline{x'})-\eta(e_s,x'),\text{ for all } s\in\{1,2,\ldots,j\}
\end{equation}
would fail is with $H=C$ and $\{j,s\}=\{i,l\}$. But such a case is impossible because $s\leq j\leq l-1$, as the rank function is bounded by $l-1$ when $H=C$. Therefore (\ref{sisucede}) holds, this controling ranks, and we only need to argue that the externality of $x'$ (respectively, $y'$) implies the externality of $y'$ (respectively, $x'$).

\smallskip
Assume first that $x'$ is a (rank-$j$) external cell of $\DF^k(G',n)$. Since $y'$ is a face of $x'$, $y'$ is a cell of $\DF^k(G',n)$. Furthermore $\eta(w,\overline{x'})=\eta(w,\overline{y'})$, by Proposition \ref{CorFaceOp}. Thus, in view of (\ref{simplificadas}), the externality of $y'$ will follow once we check that
\begin{equation}\label{externabilidad}
\eta(v_s,\overline{y'})-\eta(e_s,y')=\eta(v_s,\overline{x'})-\eta(e_s,x')
\end{equation}
holds for $s\in\{1,i+1\}$. But, according to Proposition \ref{CorFaceOp}, if (\ref{externabilidad}) fails for $s\in\{1,i+1\}$, then $H=C$ and $\{j,s\}=\{i,l\}$, so that in fact $j=i$ and $s=l=i+1$. But then $rank(x')=j=l-1$, which contradicts Proposition \ref{RankRestC}.

\smallskip
Assume next that $y'$ is a (rank-$j$) external cell of $\DF^k(G',n)$. Since $\widehat{y'}$ is a face of\tilde$\widehat{x'}$, it suffices to show that $x'$ is a cell of  $\DF^k(G',n)$ as, then, the externality of\tilde$x'$ follows from that of $y'$ and in view of Lemma \ref{ExtCell}. We thus prove that $x'$ has no $k$-fold collisions. Actually, since $y'$ has no $k$-fold collisions, a possible $k$-fold collision in $x'$ would have to be in $v_j$, as a result of expanding $y'_m=\iota(e_j)$ to $x'_m=e_j$. All together, we only have to check the inequality
\begin{equation}\label{ahorasilounicoademostrar}
\eta(v_j,\overline{x'})\leq k-1.
\end{equation}
With this in mind, start by noticing that the case $s=j$ in (\ref{sisucede}) and the fact that $y'$ is a rank-$j$ cell imply
$\eta(v_j,\overline{x'})-\eta(e_j,x')=\eta(v_j,\overline{y'})-\eta(e_j,y')=0$, i.e., $\eta(v_j,\overline{x'})=\eta(e_j,x')$. This and the fact that $x'$ has exactly one more cell $e_j$ than $y'$ yields (\ref{ahorasilounicoademostrar}), because
\begin{align*}
    \eta(e_j,x')=\eta(e_j,y')+1\leq \eta(e_j,y')+\eta(\iota(e_j),y')\leq \eta(\iota(e_j),\overline{y'})\leq k-1,
\end{align*}
as $y'$ has no $k$-fold collisions.
\end{proof}

We are now in position to construct our perfect pairing $W=\{W_j\}_j$. In what follows $x'=(x'_1,\ldots,x'_n)$ is a rank-$j$ external cell of $\DF^k(G',n)$. Based on Lemma \ref{paraelcampo}, we let $t(x')$ be the largest $t\in\{1,2,\ldots,n\}$ such that $x'_t\in\{\iota(e_j),e_j\}$, i.e.,
$$
t(x'):=\max\left(\iota(e_j)(x')\cup e_j(x')\rule{0mm}{4mm}\right).
$$
If $t(x')\in e_j(x')$, we construct the pair $(y',x')\in W_j$, where $y'$ is obtained from $x'$ by replacing $x'_{t(x')}=e_j$ by $\iota(e_j)$. Otherwise, $t(x')\in \iota(e_j)(x')$ and we construct the pair $(x',z')\in W_j$, where $z'$ is obtained from $x'$ by replacing $x'_{t(x')}=\iota(e_j)$ by $e_j$. We also use the notation $t_j(x')$, instead of the simpler $t(x')$, in order to stress the fact that $rank(x')=j$. In view of Corollary \ref{FieldWD}, the construction above is well-defined and determines a (partial) matching $W=\{W_j\}_j$ on the cells of $\DF^k(G',n)$ in such a way that a pair of external cells share rank and $t$-value, and whose unpaired cells assemble the subcomplex $Y$ homeomorphic to $\DF^k(G,n)$ in Theorem \ref{StabDiscThm}.

\smallskip
Following standard notation in discrete Morse theory, the Hasse diagram\tilde$H$ of the face poset of $\DF^k(G',n)$ is thought of as an oriented graph, with oriented edges $a\searrow b$ whenever\tilde$b$ is a codimension-1 face of $a$. The $W$-modified Hasse diagram $H_W$ is the oriented graph obtained from $H$ by reversing the orientation of edges corresponding to pairs of $W$. We then write $y'\nearrow x'$ and $x'\nearrow z'$ instead of $(y',x')\in W$ and $(x',z')\in W$, reserving the symbol ``$\searrow$'' for non-reversed arrows in $H_W$. Furthermore, any cell which is the head (tail) of a reversed Hasse arrow is called collapsible (redundant), while unpaired cells (namely those of $Y$) are called critical.

\section{Acyclicity}
Throughout this final section we assume that $G$ is $(k,n)$-sufficiently subdivided, so that the rank function and the resulting pairing $W$ are well-defined. Our final task, addressed below, is a proof of the acyclicity of $W$. With this in mind, note that a potential cycle
$$\alpha_0\nearrow\beta_0\searrow\alpha_1\nearrow\beta_1\searrow\dots \searrow\alpha_n\nearrow\beta_n\searrow\alpha_0
$$
in $H_W$ would consist exclusively of external (non-critical) cells $\alpha_j$ and $\beta_j$. The following key result is the basis for describing the dynamics ---and, eventually,  non-existence--- of these alleged cycles.

\begin{lemma} \label{AcyclicLemma}
Let $x'=(x'_1,\ldots,x'_n)$ and $y'=(y'_1,\ldots,y'_n)$ be external cells of respective ranks $j$ and $j-h$, with $x'$ collapsible and $y'$ redundant. Assume that $x'\searrow y'$ is an edge of\tilde$H_W$, with $y'$ obtained from $x'$ by replacing an edge $f=x'_t$ of $x'$ by a vertex $u=y'_t$ of $f$.
\begin{itemize}
\item[\emph{(A)}] If $h<0$, then $f=e_{j}$ and $u=v_{j}$.
\item[\emph{(B.1)}] If $h=0$ and $j\in\{1,i+1\}$, then $u=w$ and $f=e_{j'}$ for $j'\in\{1,i+1\}-\{j\}$.
\item[\emph{(B.2)}] If $h=0$ and $j\not\in\{1,i+1\}$, then $u=v_{j-1}$ and $f=e_{j-1}$.
\item[\emph{(C.1)}] If $h=1$ and $j=i+1$, then $u\neq v_i$, $v_i$ is a vertex of $f$, and either $f$ is not an edge of $H'$ or, else, $f=e_l$ and $i<l-2$ (note that $H=C$ is forced in the latter case).
\item[\emph{(C.2)}] If $h=1$ and $j\neq i+1$, then
$f=e_{j}$ and $u=v_{j}$.
\end{itemize}
When $h\geq2$, one of the following three mutually-disjoint situations must hold:
    \begin{itemize}
        \item[\emph{(D.1)}] $f=e_{s}$ and $u=v_{s}$ where $i+1\neq s=rank(y')+1\leq j-1$.
        \item[\emph{(D.2)}] $f$ is not an edge of $H'$, $u\neq v_i$, $v_i$ is a vertex of $f$, and $rank(y')=i\leq j-2$.
        \item[\emph{(D.3)}] $f=e_l$ with $j+1<l$, $u\neq v_i$, $rank(y')=i\leq j-2$ and $H=C$.
    \end{itemize}
\end{lemma}
\begin{proof}
{\bf(A)} Assume $h<0$, so that $rank(y')>j=rank(x')$. Then
\begin{equation}\label{Auno}
v_j(\overline{y'})-e_j(y')\neq \varnothing=v_j(\overline{x'})-e_j(x'),\end{equation}
which forces $v_j(\overline{y'})-e_j(y')=\{t\}$, in view of Remark \ref{aclaracion}, so that $v_j=u$. In particular $v_j$ is a vertex of $f$, and then the equality in (\ref{Auno}) forces $f=e_j$.

\medskip\noindent{\bf(B)}
Assume $h=0$, so that $rank(y')=j=rank(x')$. We only prove the case having $j=i+1$, as the other two cases ($j=1$ and $j\not\in\{1,i+1\}$) are proved in an entirely parallel way. For instance, when $j\not\in\{1,i+1\}$, the argument below applies word for word replacing the subindices $1$ and $i+1$ by $j-1$ and $j$, respectively, and replacing $w$ by $v_{j-1}$. So, we assume in addition $j=i+1$. If $f\neq e_{1}, e_{i+1}$, then $w$ is not a vertex of $f$, so $w\neq u$ and we get $w(y')=w(x')$ and $e_{i+1}(y')=e_{i+1}(x')$. Hence $t_{i+1}(y')=t_{i+1}(x')$ with $x'_{t_{i+1}(x')}=y'_{t_{i+1}(y')}$, which is impossible, since $y'$ is redundant and $x'$ is collapsible. Therefore $f\in\{e_{1}, e_{i+1}\}$. We show next that the option $f=e_{i+1}$ also leads to a contradiction.

Assume  $f=e_{i+1}$. Then $u\neq v_{i+1}$, for otherwise $t\in v_{i+1}(\overline{y'})- e_{i+1}(y')$, which cannot happen as $y'$ has rank $i+1$. But $u$ is a vertex of $f=e_{i+1}$, so that $u=w$. The last two equalities then show that $t$ lies in $e_{i+1}(x')\cap w(y')$. Thus, by Remark \ref{aclaracion}, 
$$
w(y')\cup e_{i+1}(y')=w(x')\cup e_{i+1}(x'),
$$
which implies in turn that $t_{\max}:=t_{i+1}(y')=t_{i+1}(x')$. Note that $t\neq t_{\max}$, for otherwise we would have the directed edge  $y'\nearrow x'$, in contradiction to the assumed directed edge $x'\searrow y'$. But then $t<t_{\max}$, and thus $x'_{t_{\max}}=y'_{t_{\max}}$ which, as above, is impossible ($y'$ is redundant and $x'$ is collapsible).

Hence $f=e_1$ and, in particular,
\begin{equation}\label{medio}
e_{i+1}(y')=e_{i+1}(x').
\end{equation}
Additionally, the equality $u=w$ must hold for, otherwise, $w(y')=w(x')$ which, together with (\ref{medio}) implies $t_{i+1}(y')=t_{i+1}(x')$ with $x'_{t_{i+1}(x')}=y'_{t_{i+1}(y')}$, again a contradiction.

\medskip\noindent{\bf(C)} Assume $h=1$, so that $rank(x')=j$ and $rank(y')=j-1$ ($j\geq2$). The latter equality gives $v_{j-1}(\overline{y'})=e_{j-1}(y')$. In particular $v_{j-1}(y')=\varnothing$ and so
\begin{equation}\label{lacomun}
u\neq v_{j-1}.
\end{equation}
Additionally, both rank hypotheses give
$v_{j-1}(\overline{y'})-e_{j-1}(y')=\varnothing\neq v_{j-1}(\overline{x'})-e_{j-1}(x')$, which forces $v_{j-1}(\overline{x'})-e_{j-1}(x')=\{t\}$. Consequently, 
\begin{equation}\label{lasdos}
\mbox{$v_{j-1}$ must be a vertex of $f$ \, and \, $f\neq e_{j-1}$.}\end{equation}
When $j\neq i+1$, conditions (\ref{lacomun}) and (\ref{lasdos}) yield in fact $f=e_j$ and $u=v_j$, as asserted. Assume then $j=i+1$, so that (\ref{lacomun}) and (\ref{lasdos}) become $u\neq v_i$ and $f\neq e_i$ with $v_i$ a vertex of $f$. Then either $f$ is not an edge of $H'$ or, else, $f=e_l$ with $H=C$. In the latter case Proposition\tilde\ref{RankWD} gives $i+1=rank(x')\leq l-1$, i.e., $i\leq l-2$. Furthermore, the latter inequality must be a strict, for otherwise $rank(x')=l-1$ implying $v_{l-1}(\overline{x'})=e_{l-1}(x')$, which is impossible as $t\not\in e_{l-1}(x')$ (because $f=e_l$) and $t\in v_{l-1}(\overline{x'})$ (because $\iota(f)=v_{l-1}$).

\medskip\noindent{\bf(D)}
Assume $h\geq2$, so that $1\leq q:=rank(y')\leq j-2$ ($j\geq3$). The argument leading to (\ref{lacomun}) and (\ref{lasdos}) now gives $v_{q}(\overline{y'})-e_q(y')=\varnothing$, $v_{q}(\overline{x'})-e_q(x')=\{t\}$ together with
\begin{equation}\label{elresumen}
u\neq v_q, \:\: f\neq e_q, \:\: \text{and \, $v_q$ is a vertex of $f$.}
\end{equation}
When $q\neq i$, (\ref{elresumen}) forces $f=e_{q+1}$ and $u=v_{q+1}$, which fits (D.1). On the other hand, when $q=i$, we must have either \emph{(a)} $f=e_l$ with $H=C$ or, else, \emph{(b)} $f$ is not and edge of $H'$. The latter situation fits (D.2) in view of (\ref{elresumen}), whereas the former situation fits (D.3) as the inequality $j+1<l$ is forced. Indeed, as in {\bf(C)} above, Proposition \ref{RankWD} gives $j=rank(x')\leq l-1$, while an equality $rank(x')=l-1$ would imply $v_{l-1}(\overline{x'})=e_{l-1}(x')$, which is impossible because, under the conditions in \emph{(a)}, $t\in v_{l-1}(\overline{x'})\setminus e_{l-1}(x')$.
\end{proof}

The following consequence of Lemma \ref{AcyclicLemma} is a partial strengthening of the rank-stability property in Proposition \ref{CorFaceOp}.
\begin{corollary} \label{ConstRank}
Let $x'=(x'_1,\ldots,x'_n)$ be a collapsible external cell with $x'_t=e_s$ for some $t\in\{1,\dots ,n\}$ and where $s$ is the rank of some external cell $z'$ (possibly $z'\neq x'$). Assume that the $n$-tuple $y'$ obtained from $x'$ by replacing the $t$-th coordinate $e_s$ of $x'$ by $\iota(e_s)$ is a redundant external cell. Then
\begin{equation}\label{rankequality}
rank(y')=rank(x').
\end{equation}
If in addition $x'\searrow y'$ in $H_W$, then $\{s,rank(x')\}\in\{1,i+1\}$.
\end{corollary}
\begin{proof}
The asserted equality (\ref{rankequality}) holds by construction when $y'\nearrow x'$, so we can safely assume $x'\searrow y'$. Set $h:=rank(x')-rank(y')$, $j:=rank(x')$, $f:=e_s$ and $u:=\iota(e_s)$. Since $u=\iota(f)$ and $f=e_s$, possibilities (A), (B.2), (C.2), (D.1) and (D.2) in Lemma \ref{AcyclicLemma} are ruled out. On the other hand, in cases (C.1) and (D.3) we would have $H=C$ and $f=e_l$, i.e., $s=l$, which is impossible by Proposition \ref{RankWD} and the assumption that $s$ is the rank of some external cell. Therefore the situation in item (B.1) holds, which completes the proof.   
\end{proof}

\begin{proposition} \label{ConstDiff}
Let $y'_1\nearrow x'_1\searrow\dots \searrow y'_\nu\nearrow x'_\nu$ be a directed path in $H_W$ consisting of external cells of rank at most $i+1$. If $i+1$ is the rank of some external cell, then $v_{i+1}(\overline{z'})-e_{i+1}(z')=v_{i+1}(\overline{z''})-e_{i+1}(z'')$, for any $z',z''\in\{y'_1,x'_1,\ldots,y'_\nu,x'_\nu\}$
\end{proposition}
\begin{proof}
For $1\leq r\leq \nu$, let $j_r$ be the common rank of $y'_r$ and $x'_r$, so that
\begin{equation}\label{upperbound}
j_r\leq i+1.
\end{equation}
By construction of the field $W$, $x'_r$ is obtained from $y'_r$ by replacing, in some coordinate, a vertex-entry $\iota(e_{j_r})$ of $y'_r$ by the edge-entry $e_{j_r}$ in $x'_r$. By Proposition \ref{CorFaceOp}, the equality 
$v_{i+1}(\overline{x'_r})-e_{i+1}(x'_r)=v_{i+1}(\overline{y'_r})-e_{i+1}(y'_r)$ holds unless $H=C$ and $l\in\{j_r,i+1\}$. But such a restriction never holds in view of Proposition \ref{RankWD}, as both $j_r$ and $i+1$ are then actual rank values. It remains to prove
\begin{equation}\label{estabilidad}
v_{i+1}(\overline{x'_r})-e_{i+1}(x'_r)=v_{i+1}(\overline{y'_{r+1}})-e_{i+1}(y'_{r+1})
\end{equation}
when $r<\nu$. With this in mind, set $h_r:=rank(x'_r)-rank(y'_{r+1})$ and let us say that, in coordinate $t_r\in\{1,2,\ldots,n\}$, $x'_r$ has an edge $f_r$ while $y'_{r+1}$ has the vertex $u_r$ of $f_r$. Depending of the actual value of $j_r$, the following explicit possibilities arise:
\begin{itemize}
\item[(I)] $j_r=i+1$, so that $h_r\geq 0$ in view of (\ref{upperbound}). Lemma \ref{AcyclicLemma} then yields one of the following situations:
\begin{itemize}
\item[(a)] $h_r=0$, $u_r=w$ and $f_r=e_1$.
\item[(b)] $h_r=1$, $u_r\neq v_i$, $v_i$ is a vertex of $f_r$, and $f_r$ is not an edge of $H'$.
\item[(c)] $h_r=1$, $u_r\neq v_i$, $f_r=e_l$, $rank(y'_{r+1})=i<l-2$, and $H=C$.
\item[(d)] $h_r\geq2$, $u_r=v_s$ and $f_r=e_s$ with $s=rank(y'_{r+1})+1\leq j_r-1=i$.
\end{itemize}
Note that neither (D.2) nor (D.3), where the condition $i\leq j_r-2$ holds, are valid options under the current assumption $j_r=i+1$. 

\item[(II)] $1<j_r<i+1$. In this case Lemma \ref{AcyclicLemma} yields one of the following situations:
\begin{itemize}
\item[(e)] $h_r<0$, $f_r=e_{j_r}$ and $u_r=v_{j_r}$.
\item[(f)] $h_r=0$, $f_r=e_{j_r-1}$ and $u_r=v_{j_r-1}$.
\item[(g)] $h_r=1$, $f_r=e_{j_r}$ and $u_r=v_{j_r}$.
\item[(\hspace{.2mm}h)] $h_r\geq2$, $f_r=e_s$ and $u_r=v_s$ where $s=rank(y'_{r+1})+1\leq j_r-1$.
\end{itemize}
As above, neither (D.2) nor (D.3) are valid options.

\item[(III)] $j_r=1$, so that $h_r\leq 0$. Lemma \ref{AcyclicLemma} now gives one of the following situations:
\begin{itemize}
\item[(i)] $h_r<0$, $f_r=e_1$ and $u_r=v_1$.
\item[(j)] $h_r=0$, $f_r=e_{i+1}$ and $u_r=w$.
\end{itemize}
\end{itemize}

In (j), equality $u_r=w$ implies $t_r\not\in v_{i+1}(\overline{y'_{r+1}})$, whereas equality $f_r=e_{i+1}$ implies $t_r\in e_{i+1}(x'_r)$. Therefore $t_r$ lies neither on the left hand-side nor on the right hand-side of (\ref{estabilidad}), so Remark \ref{aclaracion} gives the asserted equality in (\ref{estabilidad}). Likewise, for all remaining cases (a)--(i), we show below
\begin{equation}\label{quizasdosultimas}
f_r\neq e_{i+1}\:\:\:\text{and}\:\:\:\left(\mbox{$v_{i+1}$ is a vertex of $f_r\:\Longleftrightarrow\:u_r=v_{i+1}$}\rule{0mm}{4mm}\right),
\end{equation}
so (\ref{estabilidad}) follows by observing that (\ref{quizasdosultimas}) implies that either the two sets in (\ref{estabilidad}) contain $t_r$ or, else, none of them contains $t_r$.

We leave it for the reader the elementary verification that the first assertion in (\ref{quizasdosultimas}) holds in all cases (a)--(i). Furthermore the implication ``$\Leftarrow$'' in the second assertion of (\ref{quizasdosultimas}) is obvious, so we focus on the opposite implication. In cases (a), (c), (f), (h) and (i) direct inspection shows that, in fact, $v_{i+1}$ is not a vertex of $f_r$. (The same conclusion holds in the case of (b), though it requires some argumentation; see below.) On the other hand, in cases (d), (e) and (g) we have $f_r=e_s$ and $u_r=v_s$ with $s\leq i$, so that, if $v_{i+1}$ is a vertex of $f_r$, then the three conditions $H=C$, $s=i$ and $i+1=l$ are forced to hold and, in view of\tilde(\ref{cerrado}), yield the condition $u_r=v_s=v_i=v_l=v_{i+1}$ asserted in (\ref{quizasdosultimas}). Lastly, the situation in\tilde(b) is slightly special in that $v_i$ and $u_r$ are the two (different) vertices of $f_r$, which is an edge outside $H'$. (In particular $u_r\neq v_{i+1}$.) But then, if $v_{i+1}$ was a vertex of $f_r$, we would necessarily have $v_{i+1}=v_i$ so, again,
\begin{equation}\label{porfin}
H=C\text{\:\:\: and \:\:\:}i=l-1.
\end{equation}
But Proposition \ref{RankWD} and the first condition in (\ref{porfin}) give $j_r=rank(x'_r)\leq l-1$ which, taking into account condition $j_r=i+1$ in (b), yields $i\leq l-2$, in contradiction to the second condition in (\ref{porfin}).
\end{proof}

The proof of Theorem \ref{StabThm} is finally complete in view of:
\begin{theorem}
The field $W$ is gradient, i.e., the $W$-modified Hasse diagram $H_W$ is acyclic. 
\end{theorem}
\begin{proof}
Suppose for a contradiction there is a cycle
\begin{equation}\label{inexistente}
y'_1\nearrow x'_1\searrow \dots \searrow y'_{\nu}\nearrow x'_{\nu}\searrow y'_{\nu+1}:=y'_1.
\end{equation}
Without loss of generality we can assume that
\begin{equation}\label{WLOG}
j:=rank(y'_1)=rank(x'_1)
\end{equation}
is the largest integer among the ranks of the cells involved in the cycle. In particular, the first pairing $y'_1\nearrow x'_1$ is obtained by replacing the vertex $\iota(e_j)$ in coordinate $t:=t_j(y'_1)$ of $y'_1$ by the edge $e_j$. Furthermore, since (\ref{inexistente}) is a cycle and since each pairing $y'_r\nearrow x'_r$ is obtained by replacement of some vertex-coordinate $\iota(e_p)$ of $y'_r$ by the corresponding edge-coordinate $e_p$ of $x'_r$, it must be the case that, at some point $x'_m\searrow y'_{m+1}$ of the path, the edge $e_j$ that was introduced in coordinate $t$ by the first pairing gets removed, necessarily replacing it by $\iota(e_j)$.
Note that $m\geq 2$ because the directed path $y'_1\nearrow x'_1\searrow y'_1$ is impossible in $H_W$. Corollary \ref{ConstRank} then gives $\{j,rank(x'_m)\}=\{1,i+1\}$.
Actually, since $rank(x'_m)\leq j$, we necessarily have
\begin{equation}\label{necessarily}
j=i+1,\:\: rank(x'_m)=1 \text{\:\:and, in particular,\:\:}\eta(v_{1}, \overline{x'_m})-\eta(e_{1}, x'_m)=0.
\end{equation}
On the other hand, since $i+1$ is the largest integer among the ranks of the cells in $y'_1\nearrow x'_1\searrow \dots \searrow y'_{m}\nearrow x'_{m}$, Proposition\tilde\ref{ConstDiff}, (\ref{WLOG}) and (\ref{necessarily}) give $$\varnothing=v_{i+1}(\overline{y'_1})-e_{i+1}(y'_1)=v_{i+1}(\overline{x'_m})-e_{i+1}(x'_m)$$ or, in terms of cardinalities
\begin{equation}\label{contrafinal1}
\eta(v_{i+1}, \overline{x'_m})-\eta(e_{i+1}, x'_m)=0.
\end{equation}
But then (\ref{necessarily}) and (\ref{contrafinal1}) yield
\begin{itemize}
\item $\eta(w,\overline{x'_m})+\eta(v_{i+1}, \overline{x'_m})-\eta(e_{i+1}, x'_m)=\eta(w,\overline{x'_m})\leq k-1$,
\item $\eta(w,\overline{x'_m})+\eta(v_{1}, \overline{x'_m})-\eta(e_{1}, x'_m)=\eta(w,\overline{x'_m})\leq k-1$,
\end{itemize}
in contradiction to Proposition \ref{spelled} and the fact that $x'_m$ is an external cell.
\end{proof}


\medskip
{\small \sc Departamento de Matem\'aticas

Centro de Investigaci\'on y de Estudios Avanzados del I.P.N.

Av.~Instituto Polit\'ecnico Nacional n\'umero 2508

San Pedro Zacatenco, M\'exico City 07000, M\'exico

{\tt oalvarado@math.cinvestav.mx}

{\tt jesus@math.cinvestav.mx}}

\end{document}